%% file: RelPD2Main.tex
\documentclass[a4paper]{article}
\input{macros}

\input{environs}

\newcommand{\U}{\ensuremath{\mathbb{U}}}
\newcommand{\cgp}[1]{\ensuremath{\overline{\gp{#1}}}}
\newcommand{\modn}[1]{\ensuremath{\quad \text{modulo }G_{#1}}}
\DeclareMathOperator{\gr}{gr}
\newcommand{\grj}[1][j]{\ensuremath{\gr_{#1}(G)}}
\DeclareMathOperator{\tg}{tg}
\DeclareMathOperator{\qt}{qt}
\DeclareMathOperator{\stab}{stab}
\DeclareMathOperator{\val}{val}
\DeclareMathOperator{\rk}{rk}

\title{Classification of pro-$p$ PD$^2$ pairs and the pro-$p$ curve complex}
\author{Gareth Wilkes}
\begin{document}
\maketitle

\input{intro}
\input{Prelims}
\input{Capping}
\input{Classification}
\input{complex}

\input{p-congtop}

\bibliographystyle{alpha}
\bibliography{RelPD2.bib} 
\end{document}

%% file: macros.tex
\usepackage{amsmath}
\usepackage{amsfonts}
\usepackage{amssymb}
\usepackage{amsthm}
\usepackage{mathtools}
\usepackage{tikz}
\usepackage{tikz-cd}
\usepackage{enumerate}
\usepackage{layout}
\usepackage{mathabx}
\usepackage{bm}
\usetikzlibrary{patterns}

\DeclareMathOperator{\id}{id}
\DeclareMathOperator{\Aut}{Aut}
\DeclareMathOperator{\Out}{Out}
\DeclareMathOperator{\Inn}{Inn}
\DeclareMathOperator{\cd}{cd}

\DeclareMathOperator{\Hom}{Hom}

\DeclareMathOperator{\Ext}{Ext}
\DeclareMathOperator{\Tor}{Tor}

\DeclareMathOperator{\ind}{ind}

\DeclareMathOperator{\im}{im}

\newcommand{\bdy}{\ensuremath{\partial}}

\newcommand{\iso}{\ensuremath{\cong}}
\newcommand{\Z}[1][]{\ensuremath{\mathbb{Z}_{#1}}}
\newcommand{\Q}{\ensuremath{\mathbb{Q}}}

\newcommand{\N}{\ensuremath{\mathbb{N}}}

\newcommand{\F}{\ensuremath{\mathbb{F}}}

\newcommand{\proP}[2][p]{\ensuremath{\widehat{#2}_{(#1)}}}

\newcommand{\gp}[1]{\ensuremath{\langle #1\rangle}}
\DeclarePairedDelimiter\gpn{\langle\! \langle}{\rangle\! \rangle}

\newcommand{\famS}[0]{\ensuremath{\mathcal{S}}}

\newcommand{\Zpof}[1]{{\ensuremath{\Z[p][\![#1]\!]}}}

\newcommand{\fkB}{\ensuremath{\mathfrak{B}}}
\newcommand{\lqt}{\backslash}

%% file: environs.tex
\newtheorem{theorem}{Theorem}[section]

\newtheorem{prop}[theorem]{Proposition}

\theoremstyle{definition}
\newtheorem{defn}[theorem]{Definition}
\newtheorem*{cnv}{Conventions}

\theoremstyle{remark}
\newtheorem*{rmk}{Remark}

\theoremstyle{plain}

\newcounter{introthmcount}
\theoremstyle{definition}

%% file: intro.tex
\begin{abstract}
We classify pro-$p$ Poincar\'e duality pairs in dimension two. We then use this classification to build a pro-$p$ analogue of the curve complex and establish its basic properties. We conclude with some statements concerning separability properties of the mapping class group.
\end{abstract}
\section*{Introduction}
A key object in the study of self-homeomorphisms of a surface $\Sigma$ is the curve complex, a simplicial complex whose simplices are isotopy classes of certain multicurves on $\Sigma$. When working with the curve complex one often simplifies matters by assuming that a class of multicurves is in some simple standard form, and asserts that this is acceptable by noting that every multicurve may be taken to a standard form by means of some self-homeomorphism of $\Sigma$. One proves this by cutting $\Sigma$ along the multicurve and using the classification of surfaces with boundary. Thus the curve complex may be viewed as built using this classification as a foundation.  

In this paper we will carry out this program in the realm of pro-$p$ groups. The notions of `relative (co)homology' and `Poincar\'e duality pair' for profinite groups were defined and developed in the author's paper \cite{Wilkes17}. In the first part of this paper we will classify those pro-$p$ group pairs which are Poincar\'e duality pairs of dimension 2. The end result (Theorem \ref{MainThm}) is closely related to (and indeed relies upon) the classification of Demushkin groups completed by Demushkin, Serre, and Labute \cite{demus61, demus63,Serre62,labute67}. 

In Section \ref{secGimel} we use this classification to classify the ways that the pro-$p$ completion $G$ of $\pi_1\Sigma$ (for $\Sigma$ a compact orientable surface) may split as a graph of pro-$p$ groups with edge stabilisers $\Z[p]$ and show that, up to automorphism, the only possible splittings are the obvious geometric ones.  Considering splittings as being dual to `simple closed curves' this analysis yields a `pro-$p$ curve complex'---a profinite space, containing the abstract curve complex, on which $\Out(G)$ has a natural action. This raises the question of what properties of automorphisms can be deduced from the action on this pro-$p$ curve complex---in particular, is it possible for a pseudo-Anosov automorphism of $\pi_1\Sigma$ to fix a point in the pro-$p$ curve complex? In Section \ref{secpCongTop} we will note how the action on the pro-$p$ curve complex gives information concerning certain separability properties of the mapping class group.

It is perhaps somewhat amusing to compare the order of this program with the classification of discrete PD$^2$ groups and PD$^2$ pairs. In the discrete realm, the PD$^2$ groups were shown to be precisely the surface groups by Bieri, Eckmann, M\"uller and Linnell \cite{BE77, EM80, EL83}. As noted in \cite[Section 11]{BE77} this implies the classification of PD$^2$ pairs via the classification of splittings of surface groups. In our present case the PD$^2$ pairs are classified first, and then applied to study splittings. 

The author would like to thank his supervisor Marc Lackenby for his support. The author was supported by the EPSRC.

\begin{cnv} In this paper we adopt the following conventions.
\begin{itemize}
\item The symbol $p$ will denote a fixed prime number. The symbol \Z[p] denotes the $p$-adic integers and $\F_p$ the field with $p$ elements.
\item Any appearance of $p^\infty$ should be interpreted as zero.
\item All groups will be pro-$p$ groups and all subgroups closed unless otherwise specified. All homomorphisms will be continuous.
\item For group elements $x$ and $y$, the conjugate $y^{-1}xy$ will be denoted $x^y$ and the commutator $[x,y]$ will mean $x^{-1}y^{-1}xy=x^{-1}x^y$.
\item Let $G$ be a pro-$p$ group and let $M$ be a finite $p$-primary $G$-module. Then $M^\ast$ will denote the Pontrjagin dual $\Hom(M, \Q_p/\Z[p])$ with action $(g\cdot f)(m)=f(g^{-1}m)$.
\item The automorphism group $\Aut(\Z[p])$ will be denoted $\U_p$.
\item For a pro-$p$ group $G$ and a homomorphism $\rho\colon G\to\U_p$ we will denote the the module with underlying abelian group $\Z[p]$ and $G$-action $\chi$ by $\Z[p](\chi)$. The module $\Z/p^m$ with action given by $\rho$ will be denoted $M_m(\rho)$ and its Pontrjagin dual by $I_m(\rho)$.
\end{itemize}
\end{cnv}

%% file: Prelims.tex
\section{Preliminaries}
\subsection{Automorphisms of \Z[p]}
In this section we set up some notation and preliminaries concerning automorphisms of the $p$-adic integers and their quotient groups. First recall the structure of the automorphism group.
\begin{theorem}[Theorem 4.4.7 of \cite{RZ00}]
Let $p$ be a prime.
\begin{itemize}
\item If $p\neq 2$ then $\U_p=\Z[p]^{\!\times}\iso C_{p-1}\times \Z[p]$ where the second factor is generated by the automorphism $1\mapsto 1+p$. 
\item $\U_2=\Z[2]^{\!\times}\iso C_{2}\times \Z[2]=\gp{-1}\times\cgp{1\mapsto 1+4}$ 
\end{itemize}
\end{theorem}
Note that for $p\neq 2$ the image of any homomorphism from a pro-$p$ group to $\U_p$ meets $C_{p-1}$ trivially. Following \cite{labute67} we give the following notation to the pro-$p$ subgroups of these automorphism groups. For $1\leq k\leq \infty$ an integer set
\[\U_p^{(k)} = \cgp{1\mapsto 1+p^k}\leq\U_p, \quad \U_2^{[k+1]}= \cgp{1\mapsto -1+2^{k+1}}\leq \U_2 \]

Let $G$ be a pro-$p$ group and let $\chi\colon G\to \U_p$ be a homomorphism. Note that since $\Z[p]$ has a unique subgroup of index $p^m$ for any $m$ there is an induced map from $G$ to $\Aut(\Z/p^m)$. We may define an invariant $q(\chi)$ to be the highest power $p^m$ of $p$ such that $G\to\Aut(\Z/p^m)$ is trivial. If this maximum does not exist (i.e.\ $\chi$ is trivial) then by convention we set $q(\chi)=p^\infty=0$. Note that for any pro-$p$ group $G$ and any $\chi\colon G\to\U_p$ the image of $G$ in $\Aut(\F_p)$ is trivial. 

One may easily show that if $q(\chi)=p^m\neq 2$ then \(\im(\chi) = \U_p^{(m)} \). If $q(\chi)=2$ then the image of $\chi$ may be either $\gp{-1}\times \U_2^{(f)}$ or $\U_2^{[f]}$ where $f\geq 2$ or $f=\infty$. Note that $\U_2^{(1)}=\U_2^{[2]}$.

\subsection{Lower central $q$-series}
Let $G$ be a pro-$p$ group and let $q=p^m$ for some $1\leq m\leq \infty$. The {\em lower central $q$-series} of $G$ is the sequence of closed subgroups of $G$ defined by
\[G_1 = G, \quad G_{j+1}=G_j^q\overline{[G_j,G]}\]
We may define 
\[\grj{} = G_j/G_{j+1}, \quad \grj[]=\sum_{j\geq 1} \grj{} \]
The fact that $[G_i, G_j]\subseteq G_{i+j}$ implies that $\grj[]$ is a graded Lie algebra over $\Z[p]/q\Z[p]$, with Lie bracket induced by the commutator on homogeneous elements of \grj[].

If $q\neq 0$ then $x\mapsto x^q$ induces a linear mapping $\phi_i\colon \grj{}\to\grj[j+1]$ for all $j$. If $\pi$ denotes an indeterminate over $\Z[p]/q\Z[p]$ we may let the polynomial ring $ \Z[p]/q\Z[p] [\pi]$ act on $\grj[]$ by defining $\pi\cdot \xi =\phi_j(\xi)$ when $\xi\in\grj{}$. By convention we extend this to the $q=0$ case by setting $\pi=0\in\Z[p]$. This makes $\grj[]$ into an algebra over $\Z[p]/q\Z[p][\pi]$. It is a Lie algebra provided $q=0$ or $p\neq 2$. When $p=2$ only $\sum_{j>1} \grj$ is a Lie algebra. This issue will not be of concern to us.

All the identities concerning \grj{} are derived from {\em commutator identites}. We will record some here for later use. The proofs are elementary. Let $u,v,w\in G$ and $k\in \Z$.
\begin{itemize}
\item $[uv,w] = [u,w][[u,w],v][v,w]$
\item $[u,vw] = [u,w][u,v][[u,v],w]$
\item for $u\in G_i, v\in G_j$ we have $[u^k,v]\equiv [u,v]^k\equiv [u,v^k]$ modulo $G_{i+j+1}.$
\end{itemize}

\subsection{PD$^n$ groups and PD$^n$ pairs}
It would be neither necessary nor expedient to give a full account of the theory of profinite PD$^n$ groups. We will give the minimal treatment that suffices for the present paper. For a proper discussion for profinite groups see \cite{SW00}, and for pro-$p$ groups see \cite[Section I.4.5]{Serre13}. Depending on your viewpoint the following is either a definition or a theorem.
\begin{defn} Let $G$ be a pro-$p$ group. Then $G$ is a {\em PD$^n$ group} if
\begin{itemize}
\item $\dim H^k(G,\F_p)$ is finite for all $k$;
\item $\dim H^n(G,\F_p)=1$; and
\item for all $k$ the cup product $H^k(G,\F_p)\times H^{n-k}(G,\F_p)\to H^n(G,\F_p)$ is a non-degenerate bilinear form.
\end{itemize}
\end{defn}
One may show (see \cite[Proposition 30]{Serre13} or \cite[Theorem 4.5.6]{SW00}) that these conditions are enough to guarantee a more extensive duality as follows. There is a unique homomorphism $\chi\colon G\to\U_p$ called the {\em orientation character} with the following properties. Let $I_p(G)=(\Z[p](\chi))^\ast$ denote $\Q_p/\Z[p]$ equipped with the $G$ action dual to $\chi$. Then $H^n(G,I_p(G))=\Q_p/\Z[p]$ and the cup product 
\[H^k(G, M)\otimes H^{n-k}(G, \Hom(M, I_p(G)))\to H^n(G, I_p(G)) \]
induces isomorphisms
\[H^k(G,M)\iso H^{n-k}(G, \Hom(M, I_p(G)))^\ast \iso H_{n-k}(G, \Z[p](\chi)\otimes M) \] 
for any finite $p$-primary $G$-module $M$. Here the last isomorphism is induced by Pontrjagin duality (see for example \cite[Proposition 6.3.6]{RZ00}), together with the identity 
\[\Hom(M, \Z[p](\chi)^\ast) = (\Z[p](\chi)\otimes M)^\ast \]

The theory of the (co)homology of a profinite group relative to a collection of subgroups was defined and studied in \cite{Wilkes17}. We will state a restricted version of the definition here and leave it to the curious reader to read further.
\begin{defn}[Definition 2.1 of \cite{Wilkes17}]
Let $G$ be a pro-$p$ group and $\famS=\{S_0,\ldots, S_b\}$ be a finite family of closed subgroups $(b\geq 0)$. Consider the short exact sequence
\[0\to \Delta_{G,\famS}\to \bigoplus_{i=0}^b \Zpof{G/S_i} \to \Z[p]\to 0 \]
where the final map is the augmentation. Then define
\[H_{\bullet+1}(G,\famS; M)=\Tor^G_\bullet(\Delta_{G,\famS}^\perp, M), \quad H^{\bullet+1}(G,\famS, A) = \Ext_G^\bullet(\Delta_{G,\famS}, A) \]
where $M$ is an inverse limit of finite $p$-primary $G$-modules, $A$ is a direct limit of such modules and $\Delta_{G,\famS}^\perp$ denotes $\Delta_{G,\famS}$ when considered as a right module via $\delta\cdot g=g^{-1}\delta$.
\end{defn}
The definition can actually be made for certain infinite families of subgroups as well. However in the case of duality pairs the family of peripheral subgroups is forced to be finite \cite[Proposition 5.4]{Wilkes17} so we will not concern ourselves with this. If we replace each $S_i$ by a conjugate $S_i^{g_i}$ then this leaves the cohomology invariant up to natural isomorphism by \cite[Proposition 2.9]{Wilkes17}. 

This cohomology theory is natural with respect to {\em maps of pro-$p$ group pairs}. If $(G,\famS)$ and $(G',\famS')$ are pro-$p$ group pairs with $\famS=\{S_0,\ldots, S_b\}$ and $\famS'=\{S'_0,\ldots, S'_{b'}\}$, then a map of pairs $(G,\famS)\to (G',\famS')$ consists of a function $f\colon\{1,\ldots,b\}\to\{1,\ldots, b'\}$ and a group homomorphism $\phi\colon G\to G'$ such that $\phi(S_i)\subseteq S'_{f(i)}$ for all $i$. In this circumstance one obtains natural maps
\[H_\bullet(G,\famS; M)\to H_\bullet(G',\famS'; M), \quad H^\bullet(G',\famS'; A)\to H^\bullet(G,\famS; A) \] 
for any $G'$-modules $M$ and $A$ as above, regarded as $G$-modules via $\phi$. See Proposition 2.6 of \cite{Wilkes17}.

Two important facts we will use freely are Pontrjagin duality
\[H_k(G,\famS; M)^\ast= H^k(G,\famS;M^\ast)\]
and the existence of a cup product
\[H^k(G,\famS; A)\otimes H^{n-k}(G, B) \to H^{n}(G,\famS; C) \]
for (direct limits of) finite $p$-primary $G$-modules $A$, $B$ and $C$ and a pairing $A\otimes B\to C$.

Again the following definition is more properly a theorem.
\begin{defn}[Theorem 5.15 of \cite{Wilkes17}] Let $G$ be a pro-$p$ group and $\famS$ a collection of closed subgroups of $G$. Then $(G,\famS)$ is a {\em PD$^n$ pair} if
\begin{itemize}
\item $\dim H^k(G, \famS; \F_p)$ is finite for all $k$;
\item $\dim H^n(G,\famS;\F_p)=1$; and
\item for all $k$ the cup product $H^k(G,\famS;\F_p)\times H^{n-k}(G,\F_p)\to H^n(G,\famS;\F_p)$ is a non-degenerate bilinear form.
\end{itemize}
\end{defn}
Again for a PD$^n$ pair $(G,\famS)$ there exits a unique {\em orientation character} $\chi \colon G\to \U_p$ with similar properties to above. Namely if $I_p(G, \famS)=\Z[p](\chi)^\ast$ denotes $\Q_p/\Z[p]$ equipped with the $G$ action dual to $\chi$, then $H^n(G,\famS; I_p(G,\famS))=\Q_p/\Z[p]$ and the cup product
\[H^k(G,\famS; M)\otimes H^{n-k}(G, \Hom(M, I_p(G,\famS)))\to H^n(G,\famS; I_p(G,\famS)) \]
with respect to the evaluation pairing induces isomorphisms
\[H^k(G,\famS;M)\iso H^{n-k}(G, \Hom(M, I_p(G,\famS)))^\ast \iso H_{n-k}(G, \Z[p](\chi)\otimes M) \] 
for any finite $p$-primary $G$-module $M$. See \cite[Section 5.1]{Wilkes17} for more details on PD$^n$ pairs.

Let us now state the immediate consequences of being a PD$^2$ pair. 
\begin{prop}
Let $(G,\famS)$ be a PD$^2$ pair with orientation character $\chi$. Then $G$ is a finitely generated free pro-$p$ group and $\famS$ is a finite family of infinite procyclic subgroups of $G$. Moreover $\chi$ vanishes on each element of $\famS$. 
\end{prop}
\begin{proof}
By \cite[Proposition 5.4]{Wilkes17}, $\famS$ is a finite family $\famS=\{S_0,\ldots, S_b\}$ with each $S_i$ a PD$^1$ group with orientation character $\chi|_{S_i}$. The only pro-$p$ PD$^1$ group is \Z[p], which is orientable and so each $S_i$ is cyclic with $\chi|_{S_i}=1$. Furthermore by \cite[Corollary 5.8]{Wilkes17} we find $\cd_p(G)=1$, so by (for example) \cite[Theorem 7.7.4]{RZ00} $G$ is a free pro-$p$ group. It is finitely generated since $H^1(G,\F_p)$ is finite.
\end{proof}

\begin{prop}\label{PropStdBasis}
Let $(G,\famS)$ be a PD$^2$ pair with $\famS=\{S_0,\ldots, S_b\}$. Fix a generator $s_0$ of $S_0$. Then the rank of $G$ is $n+b$ for some $n\geq 0$ and we may choose a basis \[\{x_1,\ldots, x_n, s_1, \ldots, s_b\}\] of $G$ such that 
\begin{itemize}
\item $s_i$ is a generator of $S_i$ for all $i$
\item $s_0\equiv s_1\cdots s_b$ modulo $G^p[G,G]$
\item $\zeta_i(s_0)=1$ for all $i$, where $\zeta_i$ is the projection $G\to\Z[p]$ obtained by killing all the elements of our chosen basis other than $s_i$.
\end{itemize}
\end{prop}
\begin{proof}
By the long exact sequence in relative cohomology and duality (Proposition 3.5, Theorem 5.6 and Proposition 5.7 of \cite{Wilkes17}) we have a commuting diagram (with coefficients in $\F_p$)
\[\begin{tikzcd}[column sep = large]
H^0(G) \ar[hookrightarrow]{r}{(1, \ldots, 1)} \ar{d}{\iso}& \bigoplus_{i=0}^b H^0(S_i) \ar{d}{\iso}& \\
H_2(G,\famS) \ar[hookrightarrow]{r}{}  & \bigoplus_{i=0}^b H_1(S_i)\ar{r} & H_1(G) 
\end{tikzcd} \]
whence $\bigoplus_{i=1}^b H_1(S_i)$ injects into $H_1(G)$. Therefore there is a basis \[{\cal B}=\{x_1,\ldots, x_n, s_1, \ldots, s_b\}\] of $G$ where $s_i$ is a generator of $S_i$ and $s_0 \equiv s_1\cdots s_b$ modulo $G^p[G,G]$. 

It follows that modulo $\overline{[G,G]}$ we have
 \[s_0\equiv \prod_{i=1}^n x_i^{p\mu_i} \cdot \prod_{i=1}^b s_i^{\lambda_i}\]  
where the $\lambda_i$ lie in $\Z[p]^{\!\times}$. Replacing each $s_i$ by the generator  $s_i^{\lambda_i}$ of $S_i$ we may guarantee that $\zeta_i(s_0)=1$ for all $i$.
\end{proof}
Note that the properties in the conclusion of the above theorem do not change if we replace each $s_i$ (and $S_i$) a conjugate of itself for $1\leq i\leq b$.

We will be needing a tool to identify the orientation character of a PD$^2$ pair. 
\begin{prop}\label{IdentifyChar}
Let $(G,\famS)$ be a PD$^n$ pair at the prime $p$ and let $\rho\colon G\to \U_p$. Then $\rho$ coincides with the orientation character $\chi$ of $(G,\famS)$ if and only if
\[|H^n(G,\famS; I_m(\rho))|= p^m\quad \text{for every $m\in \N$} \] 
\end{prop}
\begin{proof}
By duality we have
\begin{multline*}
H^n(G,\famS; I_m(\rho))^\ast \iso \Hom_G(I_m(\rho), I_p(G,\famS))=\\ \Hom_G(\Z[p](\chi), M_m(\rho)) = \Hom_G(M_m(\chi), M_m(\rho)) 
\end{multline*}
Now if $\chi=\rho$ then every group homomorphism from $\Z/p^m$ to itself is an element of $\Hom_G(M_m(\chi), M_m(\rho))$ so this group has size $p^m$. Conversely if the size is $p^m$ then the identity map on $\Z/p^m$ is an element of $\Hom_G(M_m(\chi), M_m(\rho))$, whence $\rho$ and $\chi$ agree modulo $p^m$. If they agree modulo $p^m$ for all $m$ then $\chi=\rho$.
\end{proof}

\subsection{Demushkin groups}
\begin{defn}
A pro-$p$ group $G$ is a {\em Demushkin group} if 
\begin{itemize}
\item $\dim H^1(G,\F_p)\leq \infty$
\item $\dim H^2(G,\F_p)=1$
\item the cup product $H^1(G,\F_p)\times H^1(G,\F_p)\to H^2(G,\F_p)$ is a non-degenerate bilinear form.
\end{itemize}
\end{defn}
The first two conditions may be rephrased as `$G$ is a finitely generated one-relator pro-$p$ group'. We call an element $r$ of a finitely generated free pro-$p$ group $F$ a {\em Demushkin word} if $F/\overline{\gpn{r}}$ is a Demushkin group.

Demushkin groups were classified by Demushkin \cite{demus61, demus63} (for $p\neq 2$), by Serre \cite{Serre62} (when $p=2$ and the rank of $G$ is odd) and by Labute \cite{labute67} for the remaining cases. The only Demushkin group which is not a PD$^2$ group is $\Z/2$, although some authors do not admit $\Z/2$ as a Demushkin group. The missing constraint to force $G$ to be PD$^2$ is that $G$ has cohomological dimension 2. This follows from the definition of a Demushkin group provided $G$ is infinite---see \cite[Section 9.1]{Serre62}. By convention we declare the orientation character of $\Z/2\Z$ to be the natural inclusion $\Z/2\to \gp{-1}\subseteq \U_2$. The classification may be stated as follows.
\begin{theorem}\label{DemusClass}
Let $F$ be a finitely generated free pro-$p$ group and let $r$ be a Demushkin word in $F$. Let $G=F/\overline{\gpn{r}}$ and let $\chi\colon G\to\U_p$ be the orientation character of $G$. Then there is a basis $x_1,\ldots,x_n$ of $F$ such that $r$ takes one of the following forms.
  \begin{itemize}
\item If $q\neq 2$ then $n$ is even and \[r= x_1^q[x_1,x_2]\cdots[x_{n-1}, x_n]\]
\item If $q=2$, $n$ is even and $\im(\chi)= \U_2^{[f]}$ for some $f\geq 2$ then \[r=x_1^{2+2^f}[x_1,x_2]\cdots[x_{n-1}, x_n]\]
\item If $q=2$, $n$ is even and $\im(\chi)= \gp{-1}\times\U_2^{(f)}$ for some $f\geq 2$ then \[r=x_1^2[x_1,x_2]x_3^{2^f}[x_3,x_4]\cdots[x_{n-1}, x_n]\]
\item If $q=2$, $n$ is odd and $\im(\chi)= \gp{-1}\times\U_2^{(f)}$ for some $f\geq 2$ then \[r=x_1^{2}x_2^{2^f}[x_2,x_3]\cdots[x_{n-1}, x_n]\]
\end{itemize}
Conversely all these forms of words are Demushkin words and the corresponding Demushkin groups have orientation characters as stated.
\end{theorem}

\subsection{Actions on trees}
Let us very briefly recall some background on actions of groups on trees, in both its discrete and pro-$p$ variants. More complete discussions may be found in \cite{Serre03} for the classical theory and in \cite{RZ00p,Ribes17} for the pro-$p$ theory. We will give as much of a unified theory as possible. Let $\mathfrak{C}$ be either the category of discrete groups or the category of pro-$p$ groups.

\begin{defn}
The {\em free product} (in $\mathfrak{C}$) of groups $G_1,\ldots, G_n\in\mathfrak{C}$ is a group $G\in\mathfrak{C}$ with maps $i_k\colon G_k\to G$ and the universal property that for every $H\in\mathfrak{C}$ and all maps $f_k\colon G_k\to H$ there is a unique $f\colon G\to H$ such that $fi_k = f_k$ for all $k$. 
\end{defn}
The free product for discrete groups is denoted by $\ast$, and the free pro-$p$ product by $\amalg$.
\begin{defn}
A {\em finite graph of groups} $(X,G_\bullet)$ in $\mathfrak{C}$ consists of a finite (directed) graph $X$, a group $G_x\in\mathfrak{C}$ for each $x\in X$, and monomorphisms $\bdy_i\colon G_e\to G_{d_i(e)}$ for $i=0,1$ and all edges $e\in EX$, where $d_0(e)$ and $d_1(e)$ are  the start and end vertices of $e$.

The {\em fundamental group} of $(X,G_\bullet)$ is the group formed as follows. Choose a maximal subtree $T$ of $X$. Take the free group $F$ in $\mathfrak{C}$ on the basis $\{t_e\mid e\in EX\}$ (that is, the free discrete group or free pro-$p$ group respectively). Form the free product of $F$ with all the vertex groups $G_v$ $(v\in VX)$ and factor out the (closed) normal subgroup generated by the relations:
\begin{itemize}
\item $t_e=1$ for all $e\in T$
\item $\bdy_1(g)= t_e^{-1} \bdy_0(g) t_e$ for $g\in G_e$
\end{itemize}

A graph of groups is {\em proper} if the natural maps from $G_x$ to the fundamental group are injective for all $x\in X$.
\end{defn}
The fundamental group is denoted by $\pi_1(X,G_\bullet)$ for discrete groups and by $\Pi_1(X,G_\bullet)$ for pro-$p$ groups. A `splitting' of a group $G$ will mean a proper finite graph of groups with fundamental group $G$. All graphs of discrete groups are proper. See \cite[Chapter 9]{RZ00} for more information on properness.

Let ${\cal G}=(X,G_\bullet)$ be a graph of discrete groups with finite base graph $X$ and vertex and edge groups $G_v$ and $G_e$ respectively. Let $G$ be the fundamental group of this graph of groups, denoted $\pi_1({\cal G})$ or $\pi_1(X,G_\bullet)$. There is a standard tree (the {\em Bass-Serre tree dual to $(X,G_\bullet)$}) $T=S(\cal G)$ on which $G$ acts, constructed as follows: the vertex (respectively, edge) set of $T$ consists of cosets of the vertex (respectively, edge) groups $G_x$ in $G$; that is, 
\[V(T) = \coprod_{x\in V(X)} G/G_x,\quad E(T) = \coprod_{e\in E(X)} G/G_e \]
with the obvious incidence maps given by inclusions $gG_e\subseteq gG_x$ when $x$ is an endpoint of $e$. Vertex stabilisers for the action of $G$ on $T$ are conjugates of the $G_x$, and the quotient graph $G\backslash T$ is $X$.

Similarly, given a graph of pro-$p$ groups $\proP{\cal G}=(X,\Gamma_\bullet)$ with fundamental group $\Gamma=\Pi_1(\proP{\cal G})=\Pi_1(X,\Gamma_\bullet)$, there is a `standard tree' $S(\proP{\cal G})$ with precisely the same formal definition as above. This is a topological graph which is a `pro-$p$ tree' in a homological sense. Again the quotient graph is $X$ and vertex stabilisers have the expected forms. 

These closely related notions for discrete and pro-$p$ groups interact in the following manner. Let $\Gamma$ be a discrete group. The {\em pro-$p$ topology} on $\Gamma$ is the topology whose neighbourhood basis at the identity consists of normal subgroups $N$ of $\Gamma$ with $[\Gamma:N]$ a power of $p$. A subset $S$ of $\Gamma$ is {\em $p$-separable} in $\Gamma$ if $S$ is closed in the pro-$p$ topology. For $\Delta$ a subgroup of $\Gamma$, we say that $\Gamma$ {\em induces the full pro-$p$ topology on $\Delta$}, if the induced topology on $\Delta$ agrees with its pro-$p$ topology. 
\begin{defn}
A graph of discrete groups ${\cal G}=(X,\Gamma_\bullet)$ is {\em $p$-efficient} if $\Gamma=\pi_1(\cal G)$ is residually $p$, each group $\Gamma_x$ is closed in the pro-$p$ topology on $\Gamma$, and $\Gamma$ induces the full pro-$p$ topology on each $\Gamma_x$.
\end{defn}
A $p$-efficient graph of discrete groups gives rise to a proper graph of pro-$p$ groups by taking the pro-$p$ completion $G_x$ of each $\Gamma_x$. In this case the fundamental group of the graph of pro-$p$ groups $(X,G_\bullet)$ is the same as the pro-$p$ completion of the fundamental group of the original graph of discrete groups $(X,\Gamma_\bullet)$. From the explicit definitions given above and the definition of $p$-efficiency, one map see that the Bass-Serre tree $T^{\rm abs}$ dual to the graph of discrete groups $(X,\Gamma_\bullet)$ naturally embeds in the pro-$p$ tree $T$ dual to the graph of pro-$p$ groups $(X,G_\bullet)$.

%% file: Capping.tex
\section{Capping off peripheral subgroups}
We will attack the classification problem by reducing the number of boundary components in order to make use of the classification of Demushkin groups as a base case. 

We will need the following basic proposition concerning non-degeneracy of bilinear forms. The proof is an exercise in undergraduate algebra.
\begin{prop}\label{BilFormsAndDirSums}
Let $V$ and $W$ be vector spaces over a field $F$, and let $\fkB\colon V\times W \to F$ be a bilinear form. Assume that we have direct sum decompositions 
\[V=V'\oplus V'', \quad W=W'\oplus W'' \]
and that $\fkB(v',w'')=0$ for all $v'\in V', w''\in W''$. Then $\fkB$ is non-degenerate if and only if the two bilinear forms 
\[\fkB'\colon V'\times W'\to F, \quad \fkB''\colon V''\times W''\to F \]
obtained by restriction from \fkB\ are non-degenerate.
\end{prop}

We will first deal with the case $b\geq 1$ as this only involves relative cohomology. 
\begin{theorem}\label{CappingOff1}
Let $G$ be a free pro-$p$ group of rank $n+b$ where $n\geq 0, b\geq 1$ and let $\famS=\{S_i\}_{i=0}^b$ be a family of cyclic subgroups of $G$. Suppose there exist generators $s_i$ for the $S_i$ such that 
\begin{enumerate}[(1)]
\item $s_1,\ldots, s_b$ freely generate a free factor of $G$ and
\item $s_0 \equiv s_1\cdots s_b$ modulo $G^p[G,G]$.
\item $S_i\cap\overline{\gpn{S_b}}=\{1\}$ for $i\neq b$
\end{enumerate}
Define $\overline{G}$ to be the quotient of $G$ by the normal subgroup generated by $S_b$, let $\phi\colon G\to \overline{G}$ be the quotient map and let $\overline{S}_i=\phi(S_i)$ for $0\leq i<b$. Let $\overline{\famS}=\{\overline{S}_i \}_{i=0}^{b-1}$. Note that $(\overline{G}, \overline{\famS})$ satisfies the conditions (1) and (2) above.

Then $(G,\famS)$ is a PD$^2$ pair if and only if $(\overline{G}, \overline{\famS})$ is a PD$^2$ pair. When they are both PD$^2$ pairs then their orientation characteristics $\chi, \overline{\chi}$ satisfy $\chi= \overline{\chi}\phi$.
\end{theorem}
\begin{rmk}
The assumption that $S_i\cap\overline{\gpn{S_b}}=\{1\}$ for $i\neq b$ is equivalent to $S_i\not\subseteq\overline{\gpn{S_b}}$ since $\overline{G}$ is free, hence torsion-free. This is automatic for $b>1$ by the conditions (1) and (2). For the case when $b=1$ and $S_0\subseteq \overline{\gpn{S_1}}$ see Proposition \ref{TwoPeripheralsTheSame}.
\end{rmk}
\begin{proof}
Extend $s_1,\ldots, s_b$ to a free basis $(x_1,\ldots, x_n, s_1, \ldots s_b)$ for $G$ and take $(x_1,\ldots, x_n, s_1, \ldots s_{b-1})$ as basis for $\overline{G}$. There is a map of pairs $(G,\famS)\to (\overline{G},\overline{\famS})$ induced by $\phi$ where we regard $S_b$ as sent to $1\in S_{b-1}$. This gives a commutative diagram of long exact sequences in relative cohomology with $\F_p$-coefficients. The set-up of the theorem and the fact that the coefficients are a field allow one to easily write down all the groups and maps in this sequence explicitly. In particular one finds that the  natural map
\[\F_p\iso H^2(\overline{G},\overline{\famS}) \stackrel{\iso}{\longrightarrow} H^2(G,\famS) \iso \F_p \]
is an isomorphism and that there are splittings
\[H^1(G,\famS) = H^0(S_b)\oplus H^1(\overline{G},\overline{\famS}), \quad H^1(G)=H^1(S_b)\oplus H^1(\overline{G}) \]
Here the maps from $H^0 (S_b)$ and $H^1(\overline{G},\overline{\famS})$ to $H^1(G,\famS)$ are the natural maps from the commutative diagram and the maps from $H^1(\overline{G})$ and $H^1(S_b)$ to $H^1(G)$ are respectively induced by $\phi$ and by the retraction $G\to S_b$ which kills all the elements in our chosen basis other than $s_b$.

Since cup products are natural with respect to maps of pairs \cite[Proposition 3.1]{Wilkes17}, the cup product on $H^1(\overline{G},\overline{\famS})\otimes H^1(\overline{G})$ is the restriction of the cup product on $H^1(G,\famS)\otimes H^1(G)$. Furthermore Equation 3.6 in Section 3.2 of \cite{Wilkes17} gives a commutative pentagon
\[\begin{tikzcd}
H^0(S_b)\otimes H^1(G)\ar{rr}\ar[hookrightarrow]{d} && H^0(S_b)\otimes H^1(S_b) \ar{d}{\smile}\\
H^1(G,\famS)\otimes H^1(G) \ar{rd}{\smile} && H^1(S_b)\ar{ld}{\iso}\\
& H^2(G,\famS) &
 \end{tikzcd}\]
This shows firstly that the restriction of the cup product on $H^1(G,\famS)\otimes H^1(G)$ to $H^0(S_b)\otimes H^1(S_b)$ via the direct sum decompositions above agrees with the cup product on $S_b$ and that the cup product of an element of $H^0(S_b)$ with an element of $H^1(\overline{G})$ vanishes. We may now apply Proposition \ref{BilFormsAndDirSums} to conclude that $(G,\famS)$ is a PD$^2$ pair if and only if $(\overline{G},\overline{\famS})$ is a PD$^2$ pair.

We move on to consideration of the orientation characters. Since $\chi(S_b)=1$ there is a unique map $\rho\colon\overline{G}\to\U_p$ such that $\chi=\rho\phi$. We must show that $\rho$ is the orientation character of $\overline{G}$. Let $m$ be a natural number. Note that our choices of bases give a splitting $G=\overline{G}\,\amalg\, S_b$. Let $\phi^!\colon \overline{G}\to G$ be the inclusion map and note that $\phi\phi^!=\id_{\overline{G}}$. Note also that $\chi|_{\overline{G}}=\rho$. By excision \cite[Theorem 4.8]{Wilkes17} the map $\phi^!$ induces an isomorphism $H^1(G, S_b; I_m(\chi))\to H^1(\overline{G},1; I_m(\rho))$. Therefore since $\phi\phi^!$ is the identity, $\phi$ induces an isomorphism in the other direction. Let $\famS'=\{S_0\ldots S_{b-1}\}$. The map of pairs $(G,\famS')\to (\overline{G},\overline{\famS})$ and Propositions 2.4 and 2.6 of \cite{Wilkes17} give a commutative diagram (with coefficients in $I_m(\chi)$ on the top row and $I_m(\rho)$ on the bottom row):
\[ \begin{tikzcd}
H^1(G,S_b)\ar{r} & H^1(\famS') \ar{r} & H^2(G,\famS) \ar{r} & 0\\
H^1(\overline{G},1) \ar{u}{\iso} \ar{r} & H^1(\overline{\famS})\ar{u}{\iso} \ar{r} & H^2(\overline{G},\overline{\famS}\sqcup 1) \ar{r}\ar{u} & 0
\end{tikzcd} \]   
which demonstrates that 
\[|H^2(\overline{G},\overline{\famS}; I_m(\rho))| = |H^2(\overline{G}, \overline{\famS}\sqcup 1; I_m(\rho))| = |H^2(G,\famS; I_m(\chi))|\]
and we are done by Proposition \ref{IdentifyChar}.
\end{proof}
\begin{prop}\label{TwoPeripheralsTheSame}
Let $G$ be a free pro-$p$ group of rank $n+1$ where $n\geq 0$ and let $\famS=\{S_0, S_1\}$ be a family of cyclic subgroups of $G$. Suppose there exist generators $s_i$ for the $S_i$ such that 
\begin{enumerate}[(1)]
\item $s_1$ generates a free factor of $G$ and
\item $s_0 \equiv s_1$ modulo $G^p[G,G]$.
\end{enumerate}
Assume further that $S_0\cap \overline{\gpn{S_1}}\neq 1$. If $(G,\famS)$ is a PD$^2$ pair then $n=0$, hence $G=S_0=S_1$. Conversely if $n=0$ and $G=S_0=S_1$ then $(G,\famS)$ is a PD$^2$ pair.
\end{prop}
\begin{proof}
Let $\overline{G}=G/\overline{\gpn{S_1}}$. By considering the commuting diagram of long exact sequences for the map of pairs $(G,\famS) \to (\overline{G}, \{1\}\sqcup\{1\})$ one may readily show that the maps
\[H^1(\overline{G}, \{1\}\sqcup\{1\})\to H^1(G,\famS),\quad H^1(\overline{G})\to H^1(G)\]
with coefficients $\F_p$ are respectively an isomorphism and an injection. Since $\overline{G}$ is free, $H^2(\overline{G}, \{1\}\sqcup\{1\})=0$ and the cup product on $(\overline{G}, \{1\}\sqcup\{1\})$ vanishes. Therefore the cup product of any element of $H^1(G,\famS)$ with the $n$-dimensional subspace $H^1(\overline{G})$ of $H^1(G)$ is zero. This cup product is non-degenerate so $n=0$ as claimed. The converse statement is a case of \cite[Proposition 6.16]{Wilkes17}.
\end{proof}
We now tackle the somewhat more fiddly case when we `cap off' the only peripheral subgroup.
\begin{theorem}\label{CappingOff2}
Let $G$ be a free pro-$p$ group of rank $n>0$ and let $S\subseteq [G,G]G^p$ be a cyclic subgroup. Let $\overline{G}=G/\overline{\gpn{S}}$ and let $\phi$ be the quotient map. Then $(G,S)$ is a PD$^2$ pair if and only if $\overline{G}$ is a Demushkin group. When this occurs the orientation characters $\chi$ and $\overline{\chi}$ of $(G,S)$ and $\overline{G}$ are related by $\chi=\overline{\chi}\phi$.
\end{theorem}
\begin{proof}
Firstly note that $\overline{G}$ is Demushkin if and only if $H^2(\overline{G}, 1; \F_p)=\F_p$ and the cup product 
\[H^1(\overline{G},1; \F_p)\otimes H^1(\overline{G}, \F_p)\to H^2(\overline{G}, 1; \F_p) \]
is non-degenerate. This equivalence follows immediately from the natural isomorphisms
\[H^1(\overline{G},1; \F_p)\iso H^1(\overline{G}, \F_p), \quad H^2(\overline{G},1; \F_p)\iso H^2(\overline{G}, \F_p) \]
Now note that in our situation we have $H^2(\overline{G};\F_p)\iso \F_p$ unless $S$ is trivial, in which case we are done. This follows since $\overline{G}$ is not free (since $H^1(\overline{G})=H^1(G)$, if $\overline G$ were free the map $\phi$ would be an isomorphism and $S$ would be trivial) so has nontrivial cohomology in dimension two (see \cite[Theorem 7.7.4]{RZ00}). Furthermore the dimension of $H^2(\overline{G})$ is at most one since $\overline{G}$ is a one-relator pro-$p$ group. See \cite[Section I.4.3]{Serre13}.

From the condition $S\subseteq [G,G]G^p$ and long exact sequences in cohomology one easily deduces that the natural map 
\[H^1(\overline{G}, 1; \F_p)\to H^1(G,S; \F_p) \]
is an isomorphism. Once we know that the corresponding map on $H^2$ is an isomorphism the first part of the theorem follows immediately. The long exact sequence implies that $H^2(G,S)$ is also isomorphic to $\F_p$. Verifying that the {\em natural} map $H^2(\overline{G}, 1)\to H^2(G,S)$ is itself an isomorphism is rather technical and we delay it to the end of the proof. 

For now we move onto verifying that the orientation characters of $(G,S)$ and $\overline{G}$ agree. Let $m$ be a natural number. We will make use of the Five Term Exact Sequence for the extension $1\to R\to G\to\overline{G}\to 1$ \cite[Corollary 7.2.5]{RZ00}. Here the relevant portion takes the form
\begin{equation}\label{5TES}
H^1(G,I_m(\rho\phi))\to H^1(R, \Z/p^m)^G\to H^2(\overline{G}, I_m(\rho))\to 0
\end{equation} 
where $R=\overline{\gpn{S}}$ and $\rho\colon\overline{G}\to\U_p$ is any homomorphism. Taking $\rho =1 $ shows that the middle term is finite (using $H^2(\overline{G}; \F_p) = \F_p$ and d\'evissage to prove that $H^2(\overline{G},\Z/p^m)$ is finite). Now consider varying $\rho$. By Proposition \ref{IdentifyChar} the modulus of the final term attains its unique maximum when $\rho$ agrees with $\overline\chi$ modulo $p^m$. On the other hand if there is a $\rho$ such that the first map in \eqref{5TES} vanishes then this $\rho$ will maximise the size of the final term, hence this $\rho$ agrees with $\overline\chi$ modulo $p^m$. In particular, since $\chi|_S$ vanishes we may write $\chi=\rho_0\phi$ for some $\rho_0$. There is now a commutative diagram induced by Poincar\'e duality 
\[\begin{tikzcd} 
& H^1(G, I_m(\chi)) \ar{dl} \ar{dd} \ar{r}{\iso} & H_1(G,S; \Z/p^m) \ar{dd}{0} \\
H^1(R, \Z/p^m)^G\ar[hookrightarrow]{dr} & & \\
& H^1(S, \Z/p^m) \ar{r}{\iso} & H_0(S; \Z/p^m)
\end{tikzcd}\]
where the fact that the rightmost map is zero follows from the long exact sequence in relative homology, and the fact that the lower left map is injective is the statement that any $G$-invariant homomorphism from $\overline{\gpn{S}}$ to $\Z/p^m$ which vanishes on $S$ is trivial. We have also used the fact that $\Z[p](\chi)\otimes I_m(\chi)$ is isomorphic to the trivial module $\Z/p^m$. It follows that the upper left map vanishes and therefore $\rho_0$ agrees with $\overline{\chi}$ and $\chi$ agrees with $\overline\chi\phi$ modulo $p^m$ for every $m$ as required.

We must now tackle the final verification we have been avoiding, that the natural map $H^2(\overline G, 1)\to H^2(G,S)$ is an isomorphism. Here and for the rest of the proof the coefficient group is assumed to be $\F_p$. We approach $H^2(\overline G)$ using the Five Term Exact Sequence as above, so it will be necessary to understand the cohomology of $R$. Firstly note that the full five term exact sequence in this case is
\[0\to H^1(\overline G)\stackrel{\iso}{\longrightarrow} H^1 (G) \to H^1(R)^G\stackrel{\tg}{\longrightarrow} H^2 (G) \to 0 \]
So that the transgression map $\tg$ is an isomorphism. By \cite[Corollary to Proposition 26]{Serre13} we know $H^1(R)^G\iso\F_p$, and as above the natural map $H^1(R)^G\to H^1(S)$ is injective, hence it is an isomorphism. 

We may also consider the long exact sequence in cohomology for the pair $(G,R)$ with coefficients in $\F_p$. This is a sequence of vector spaces so applying any additive functor---for instance the functor $(-)^G$ taking $G$-invariants---preserves exactness. From the Five Lemma and the data already collected we find that the natural map from the resulting sequence to the long exact sequence for $(G,S)$ is an isomorphism. We now have a diagram of maps
\[\begin{tikzcd} 
H^1(R)^G \ar[equal]{r} \ar{d}[swap]{\tg}{\iso} & H^1(R)^G \ar{r}[swap]{\delta}{\iso}\ar{d}{\iso} & H^1 S\ar{d}{\iso}\\
H^2 (\overline{G}) \ar{r}{\phi^\ast} & H^2(G,R)^G \ar{r}{\iso} & H^2(G,S)
\end{tikzcd}\]
where the right hand square commutes. Here the bottom left horizontal map is induced by the map of pairs $(G,R)\to(\overline{G}, 1)$ and the isomorphism $H^2(\overline{G},1)=H^2(\overline{G})$. Our aim is therefore to show that the composition along the bottom row is an isomorphism. It only remains to show that the left hand square actually commutes (in fact it will turn out that it commutes up to sign, but this is enough). This verification takes place on the level of chain complexes. 

For a pro-$p$ group $H$ and a subgroup $K$, let $C_\bullet(H)$ be the (inhomogeneous) bar resolution of $H$ over \Z[p] (see \cite[Section 6.2]{RZ00}). For our purposes it is enough to recall that the lower stages are given by
\[\Zpof{H}[\![H^2]\!] \stackrel{d_1}{\longrightarrow} \Zpof{H}[\![H]\!] \stackrel{d_0}{\longrightarrow} \Zpof{H}\]
where the boundary maps are given on the basis elements $[h_1,h_2]\in H^2$ and $[h]\in H$ by 
\[d_1([h_1, h_2]) = h_1[h_2] - [h_1h_2] + [h_1], \quad d_0([h]) = h-1 \] 
Furthermore there is a natural chain map 
\[\ind^K_H(C_\bullet(K)) \hookrightarrow C_\bullet(H)\]
which is an isomorphism on the zero level. Each quotient module $C_\bullet(H,K)$ is projective; indeed it is a free $\Zpof{H}$-module on the pointed profinite space obtained by collapsing the closed subset $K^n\subseteq H^n$ to a point. See \cite[Section 5.2]{RZ00} for information on free modules on pointed profinite spaces. The long exact sequence in homology for the short exact sequence of chain complexes
\[0\to \ind^K_H(C_\bullet(K)) \longrightarrow C_\bullet(H)\stackrel{\qt}{\longrightarrow} C_\bullet(H,K) \to 0\]
shows that the final term is a resolution for $\Delta_{H,K}$ (shifted by one degree).

In our case of interest this gives us standard chain complexes which may be used to compute the various cohomology groups in play. In particular we have an identification 
\[C_\bullet(\overline{G}, 1)\equiv C_{\bullet+1}(\overline G) \]
and the obvious map (of $G$-modules) 
\[C_\bullet(G)\longrightarrow C_\bullet(G,R)\stackrel{\qt}{\longrightarrow} C_\bullet(\overline{G}) \] 
gives the map $H^\bullet(\overline G)\to H^\bullet(G,R)$.

Now let $\zeta\colon R\to \F_p$ be a $G$-invariant homomorphism giving a generator of $H^1(R)^G$. Let $\sigma\colon \overline{G}\to G$ be a continuous section of the quotient map such that $\sigma(1)=1$ and let $\pi(g)= g\sigma(gR)^{-1}$ be the retraction $G\to R$ so defined. The section $\sigma$ exists by \cite[Proposition 2.2.2]{RZ00}. The transgression $\tg(\zeta)$ is defined in the following manner (see \cite[Section 1.4]{labute67}). The cochain \[\zeta\pi d_2\colon C_2(G)\to \F_p\]
may be shown to factor through $C_2(\overline G)$, and this cochain is $\tg(\zeta)$ by definition. It therefore also factors through $C_2(G,R)$ giving a cochain representing $\phi^\ast\tg(\zeta)$.

Now let us compute the map $\delta\colon H^1(R)^G\to H^2(G,R)$. Applying the Horseshoe Lemma (see \cite[Lemma 2.2.8]{weibel95}) gives the commuting diagram of (partial) resolutions 
\[\begin{tikzcd}[ampersand replacement = \&, row sep = huge]
C_2(G,R)\ar{r}\ar{d}{d_1} \& C_2(G,R)\oplus C_1(G) \ar{r}\ar{d}{{\begin{pmatrix} d_1 & \qt \\ 0 & d_0\end{pmatrix}}}  \&C_1(G)\ar{d}{d_0}\\
C_1(G,R)\ar{r}\ar{d}{d_0} \& C_1(G,R)\oplus C_0(G) \ar{r}\ar{d}{(i\circ d_0)\oplus \epsilon} \&C_0(G)\ar{d}{\epsilon}\\
\Delta_{G,R}\ar{r}{i} \& \Zpof{G/R} \ar{r} \& \Z[p]
\end{tikzcd}\]
The central column has a chain map to the standard resolution given by 
\[\begin{tikzcd}[ampersand replacement = \&]
C_2(G,R) \oplus C_1(G) \ar{d}{\alpha\oplus \beta} \ar{r} \& C_1(G,R)\oplus C_0(G) \ar{r}\ar{d}{\nu \oplus \id} \& \Zpof{G/R} \ar[equal]{d}\\
\Zpof{G}[\![R]\!] \ar{r} \& \Zpof{G} \ar{r} \& \Zpof{G/R}
\end{tikzcd}\]
where $\nu([g])=\sigma(gR)-1$ and
\[\alpha([g_1,g_2]) = g_1g_2[\tilde\pi(g_2)]-g_1g_2[\tilde\pi(g_1g_2)]+g_1[\tilde\pi(g_1)] - g_1[1] \]
where $\tilde\pi(g) = g^{-1}\sigma(gS) = (\pi(g)^{-1})^g$ for some map $\beta$ (which is irrelevant to what follows). We now readily see that $\delta(\zeta)=\zeta\alpha\colon C_2(G,R)\to\F_p$ is induced by the cochain $\zeta\tilde\pi d_2\colon C_2(G)\to \F_p$. As $\zeta\tilde\pi=-\zeta\pi$ we are done.
\end{proof}

%% file: Classification.tex
\section{Classification}
In this section we will prove the main classification theorem. 
\begin{defn}
Let $(G,\famS)$ be a PD$^2$ pair with $\famS=\{S_0,\ldots, S_b\}$. A {\em basis for $G$ in standard form} (or simply {\em standard form basis}) is an ordered free generating set 
\[\mathcal{B}=\{x_1,\ldots,x_n,s_1,\ldots,s_b\}\]
such that $s_i$ generates some conjugate of $S_i$ for each $i$.
\end{defn}
Note that when $b\geq 1$ any standard form basis of $G/\overline{\gpn{S_b}}$ is the image of the first $n+b-1$ elements of some standard form basis of $G$.
\begin{defn}
For integers $n$ and $b$ let $G$ be a free group of rank $n+b$ on a generating set \[\mathcal{B}=\{x_1,\ldots,x_n,s_1,\ldots,s_b\}\]
 Let $\chi\colon G\to \U_p$ be a homomorphism vanishing on the $s_i$. Define the {\em standard word} $r_1(n,b,\chi; \mathcal{B})\in G$ according to the following cases.
\begin{itemize}
\item If $n$ is even and $q(\chi)\neq 2$ then \[r_1(n,b,\chi; \mathcal{B})= x_1^q[x_1,x_2]\cdots[x_{n-1}, x_n]s_1\cdots s_b\]
\item If $q(\chi)=2$, $n$ is even and $\im(\chi)= \U_2^{[f]}$ for some $f\geq 2$ then \[r_1(n,b,\chi; \mathcal{B})=x_1^{2+2^f}[x_1,x_2]\cdots[x_{n-1}, x_n]s_1\cdots s_b\]
\item If $q(\chi)=2$, $n$ is even and $\im(\chi)= \gp{-1}\times\U_2^{(f)}$ for some $f\geq 2$ then \[r_1(n,b,\chi; \mathcal{B})=x_1^2[x_1,x_2]x_3^{2^f}[x_3,x_4]\cdots[x_{n-1}, x_n]s_1\cdots s_b\]
\item If $q(\chi)=2$, $n$ is odd and $\im(\chi)= \gp{-1}\times\U_2^{(f)}$ for some $f\geq 2$ then \[r_1(n,b,\chi; \mathcal{B})=x_1^{2}x_2^{2^f}[x_2,x_3]\cdots[x_{n-1}, x_n]s_1\cdots s_b\]
\end{itemize}
If none of these cases holds we leave $r_1(n,b,\chi,\mathcal{B})$ undefined. Note that whether or not $r_1$ is defined depends only on $n$ and $\im(\chi)$, not on $b$ or $\cal B$. When the basis $\cal B$ or the invariants $(n,b,\chi)$ are clear from context we will omit them from the notation.
\end{defn}
The main theorem of this paper is the following classification theorem.
\begin{theorem}\label{MainThm}
Let $G$ be a free pro-$p$ group of rank $n+b$ and $\famS=\{S_0,\ldots, S_b\}$ be a finite family of closed procyclic subgroups where $b\geq 0$, $n\geq 0$ and $n+b\geq 1$. Fix a generator $s_0$ of $S_0$. Then $(G,\famS)$ is a PD$^2$ pair with orientation character $\chi$ if and only there exists a standard form basis ${\cal B}=\{x_1,\ldots,x_n,s_1,\ldots,s_b\}$ of $G$ such that $s_0=r_1(n,b,\chi;\cal B)$.
\end{theorem}
\begin{rmk}
Implicit in the final sentence is the statement that $r_1(n,b,\chi)$ is not undefined.
\end{rmk}
The backwards implication---that these forms of words do genuinely give PD$^2$ pairs---follows immediately by an induction using Theorems \ref{CappingOff1} and \ref{CappingOff2} and Proposition \ref{TwoPeripheralsTheSame} taken together with the classification of Demushkin groups (Theorem \ref{DemusClass}).

The forwards implication will proceed in two steps. First we will use the results of the previous section to find a basic form for the word $s_0$. Then we will use the method of successive approximation to erode this word down to one of the forms above. 

From now on let $G$ be a free pro-$p$ group of rank $n+b$ and $\famS=\{S_0,\ldots, S_b\}$ be a finite family of closed procyclic subgroups where $b\geq 0$, $n\geq 0$ and $n+b\geq 1$ and assume $(G,\famS)$ is a PD$^2$ pair with orientation character $\chi$. Set $q=q(\chi)$ and let $(G_j)_{j\geq 1}$ be the lower central $q$-series of $G$. Fix the generator $s_0$ of $S_0$.
\begin{prop}\label{ModuloG3}
For some basis ${\cal B}$ of $G$ in standard form we have
\[s_0\equiv r_1(n,b,\chi; {\cal B})\modn{3}  \]
(and $r_1(n,b,\chi)$ is well-defined).
\end{prop}
\begin{proof}
We prove this by induction on $b$. The base cases are $b=0$---which holds by Theorem \ref{CappingOff2} and Theorem \ref{DemusClass}---and $b=1, n=0$, which holds by Proposition \ref{TwoPeripheralsTheSame}. 

Suppose that $b>1$ and that $n\neq 0$ if $b=1$. Then we may apply Theorem \ref{CappingOff1} and an induction hypothesis to show that for some standard form basis $\cal B$ we have
\[s_0\in r_1(n,b-1,\chi; {\cal B})\overline{\gpn{s_b}}G_3  \]
Let $y$ be some element of $\overline{\gpn{s_b}}$ such that 
\[s_0\equiv r_1(n,b-1,\chi; {\cal B})y\modn{3} \]
By Proposition \ref{PropStdBasis}, replacing $s_b$ by a power $s_b^\mu$ for some $\mu\in \Z[p]^{\!\times}$ we may assume that $\zeta_b(y)=1$ where $\zeta_b$ is the projection of $G$ onto $S_b$ given by killing the other elements of $\cal B$. We claim that modulo $G_3$ we may write $y$ as an element of the form $s_b^g$ for some $g\in G$.  

Consider an expression of $G/G_3$ as an inverse limit of finite $p$-group quotients $\varprojlim P_k$. Let $\psi_k\colon G/G_3\to P_k$ be the quotient map. For each $k$ we have $\psi_k(y)\in \overline{\gpn{\psi_k(s_b)}}$ and since $\zeta_b(y)=1$ we have $\psi_k(y)\equiv\psi_k(s_b)$ modulo $[P_k,P_k]$. 

Now since $P_k$ is finite, we have $\psi_k(y)\in\gpn{\psi_k(s_b)}$. That is, there are $g_i\in P_k$ and $m_i\in \Z$ ($1\leq i\leq r$) such that
\begin{eqnarray*}
\psi_k(y) &=& \prod_{i=1}^r (\psi_k(s_b)^{m_i})^{g_i}\\
&=& \prod_{i=1}^r \psi_k(s_b)^{m_i}[\psi_k(s_b)^{m_i}, g_i]\\
&=& \psi_k(s_b)^{\sum m_i} \prod_{i=1}^r [\psi_k(s_b), g_i^{m_i}]\\
&=& \psi_k(s_b)^{\prod g_i^{m_i}} 
\end{eqnarray*}
Here we have used the fact that $\psi_k(y)\equiv\psi_k(s_b)$ modulo $[P_k,P_k]$ in the final line to show that $\psi_k(s_b)^{\sum m_i}=\psi_k(s_b)$. We have made frequent use of commutator identities and the fact that $P_k$ was a quotient of $G/G_3$ hence all third commutators in $P_k$ vanish.

Hence $\psi_k(y)$ lies in the conjugacy class $\psi_k(s_b)^{P_k}$ of $s_b$ for all $k$. Since these conjugacy classes form a surjective inverse system with inverse limit $s_b^G$, there does indeed exist $g\in G$ such that $y=s_b^g$. Now replace $s_b$ by $s_b^g$ so that $s_0\equiv r_1(n,b-1,\chi)s_b=r_1(n,b,\chi)$ modulo $G_3$. This concludes the proof.
\end{proof}

We will now use the method of successive approximation to improve congruence modulo $G_3$ to equality.  We will first examine the $q\neq 2$ case to illustrate the method. The $q=2$ case simply requires more notation, together with some extra tidying up at the end.
\begin{proof}[Proof of Theorem \ref{MainThm}, $q\neq 2$ case]
Recall from Theorem \ref{DemusClass} that if $q\neq 2$ then $n=2N$ is even. If $q\neq 0$ then $n>0$. Let ${\cal B}=\{x_1,\ldots,x_n,s_1,\ldots,s_b\}$ be a basis of $G$ in standard form. If $\underline{t}=(t_1,\ldots, t_{n+b}) \in (G_{j-1})^{n+b}$ for $j\geq 3$ then we may define
\[{\cal B}'={\cal B}'(\underline{t})=\{x_1t_1,\ldots,x_nt_n,s_1^{t_{n+1}},\ldots,s_b^{t_{n+b}}\} \]
which is also a basis in standard form. One easily verifies that
\[r_1({\cal B}')\equiv r_1({\cal B})\modn{j} \]
so that there is a unique element $d_{j-1}^{\cal B}(\underline{t})\in G_n$ such that $r_1({\cal B}')= r_1({\cal B})d_{j-1}^{\cal B}(\underline{t})$. 

Let the images of $x_i$ and $s_i$ in \grj[1] be denoted $\xi_i$ and $\sigma_i$, let the image of $t_i$ in \grj[j-1] be $\tau_i$ and let the image of $\underline{t}$ in $\grj[j-1]^{n+b}$ be $\underline{\tau}$. Then the image of $d_{j-1}^{\cal B}(\underline{t})$ in \grj{} is 
\[\delta_{j-1}^{\cal B}(\underline{\tau}) = \pi\cdot\tau_1+\binom{q}{2}[\tau_1,\xi_1]+ \sum_{i=1}^N\big([\tau_{2i-1}, \xi_{2i}]+[\xi_{2i-1}, \tau_{2i}]\big) + \sum_{i=1}^b[\sigma_i, \tau_{n+i}] \]
as one may readily compute using commutator identities. This $\delta^{\cal B}_{j-1}$ is a $\Z[p]/q\Z[p]$-linear homomorphism $\grj[j-1]^{n+b}\to \grj{}$. Up to certain minus signs in the $\sigma_i$ terms and reordering some coordinates of $\grj[j-1]^{n+b}$ this is in fact the same as the map $\delta_{j-1}$ described in \cite[Proposition 5]{labute67}. Hence by that proposition we have $\im(\delta_{j-1}^{\cal B})=\grj$.

Therefore suppose that we have chosen a standard form basis ${\cal B}_{j-1}$ such that 
\[s_0 \equiv r_1({\cal B}_{j-1})\modn{j-1} \]
Then $s_0=r_1({\cal B}_{j-1})z$ for some $z\in G_j$. There exists $\underline{t} \in (G_{j-1})^{n+b}$ such that $\delta_{j-1}^{{\cal B}_{j-1}}(\underline{t})$ equals the image of $z$ in \grj. Form the new standard form basis ${\cal B}_{j}={\cal B}_{j-1}'(\underline{t})$ as above. Then 
\[s_0 \equiv r_1({\cal B}_j)\modn{j} \]
Since $\bigcap G_j=\{1\}$ and since ${\cal B}_{j-1}$ and ${\cal B}_j$ agree modulo $G_j$, for each $i$ the sequence comprising the $i^{\rm th}$ elements of the ${\cal B}_j$ converges to a limit in $G$ as $j\to \infty$. One may easily see that the set of limits $\cal B$ is a standard basis for $G$ and $s_0=r_1(\cal B)$ as required.
\end{proof}
We now move on to the $q=2$ case. The principle is the same, but the notation is more involved.
\begin{proof}[Proof of Theorem \ref{MainThm}, $q=2$ case]
Let ${\cal B}=\{x_1,\ldots,x_n,s_1,\ldots,s_b\}$ be a basis of $G$ in standard form. For $j\geq 3$, $\underline{t}=(t_1,\ldots, t_{n+b}) \in (G_{j-1})^{n+b}$ and $\underline{\epsilon}\in \{0,1\}^b$ and  define
\[{\cal B}'={\cal B}'(\underline{t}, \underline{\epsilon})=  \{x_1t_1,\ldots,x_nt_n,(s_1^{1+2^{j-1}\epsilon_1})^{t_{n+1}},\ldots,(s_b^{1+2^{j-1}\epsilon_b})^{t_{n+b}}\} \]
which is again a basis in standard form.

Again we have $r_1({\cal B}')\equiv r_1({\cal B})$ modulo $G_j$, so that there exists a unique element $d^{\cal B}_{j-1}(\underline{t},\underline{\epsilon})\in G_j$ such that 
\[r_1({\cal B}')= r_1({\cal B})d^{\cal B}_{j-1}(\underline{t},\underline{\epsilon})\]
One may readily compute that the image of $d^{\cal B}_{j-1}(\underline{t},\underline{\epsilon})$ in \grj{} is
\begin{multline*}
\delta_{j-1}^{\cal B}(\underline{\tau}) = \pi\cdot\tau_1+[\tau_1,\xi_1]+ \sum_{i=1}^N\big([\tau_{2i-1}, \xi_{2i}]+[\xi_{2i-1}, \tau_{2i}]\big)\\ + \sum_{i=1}^b\big([\sigma_i, \tau_{n+i}] + \epsilon_i \pi\cdot \sigma_i \big)
\end{multline*}
when $n=2N$ is even and
\begin{multline*}
\delta_{j-1}^{\cal B}(\underline{\tau}) = \pi\cdot\tau_1+[\tau_1,\xi_1]+ \sum_{i=1}^N\big([\tau_{2i}, \xi_{2i+1}]+[\xi_{2i}, \tau_{2i+1}]\big)\\ + \sum_{i=1}^b\big([\sigma_i, \tau_{n+i}] + \epsilon_i \pi\cdot \sigma_i \big)
\end{multline*}
when $n=2N+1$ is odd---using the same notation as in the $q\neq 2$ case. For either parity of $n$, \cite[Proposition 5]{labute67}---with minor notational tweaks---shows that the image of $\delta_{j-1}^{\cal B}$ together with the elements $\pi^{j-1}\cdot \xi_i$ generates \grj{}. Actually the cited proposition is slightly more precise than this, but we won't need this precision here as it will be folded into the last step of the proof.

As one final piece of notation, for $\underline{\lambda}\in (4\Z[2])^n$ set 
\[e({\cal B}, \underline{\lambda}) = x_1^{\lambda_1}\cdots x_n^{\lambda_n} \]
Note that for any $\underline{t}$, $\underline{\epsilon}$ as above we have
\[ e({\cal B}'(\underline{t}, \underline{\epsilon}), \underline{\lambda})\equiv e({\cal B}, \underline{\lambda}) \modn{j+1}\]
where the extra precision modulo $G_{j+1}$ follows since the $\lambda_i$ are divisible by 4.

Now suppose that $j\geq 3$ there is a basis ${\cal B}_j$ of $G$ in standard form and some $\underline{\lambda}_{j-1}\in (4\Z[2])^n$ such that 
\[s_0 = e({\cal B}, \underline{\lambda})r_1({\cal B})z \]
where $z\in G_j$. This holds for $j=3$ by Proposition \ref{ModuloG3}. From above there exist $\underline{t}$, $\underline{\epsilon}$ and $\underline{\alpha}$ such that 
\[z\equiv x_1^{2^{j-1}\alpha_1} \cdots x_n^{2^{j-1}\alpha_n} d^{\cal B}_{j-1}(\underline{t}, \underline{\epsilon}) \modn{j+1}\]
If ${\cal B}_j = {\cal B}'_{j-1}(\underline{t}, \underline{\epsilon})$ then we have
\begin{eqnarray*}
s_0&\equiv& e({\cal B}, \underline{\lambda}_{j-1})r_1({\cal B})x_1^{2^{j-1}\alpha_1} \cdots x_n^{2^{j-1}\alpha_n} d^{\cal B}_{j-1}(\underline{t}, \underline{\epsilon}) \\
&\equiv& e({\cal B}_j, \underline{\lambda}_{j-1}+2^{j-1}\underline{\alpha})r_1({\cal B}_j)
\end{eqnarray*}
modulo $G_{j+1}$. As before the ${\cal B}_j$ converge to a standard basis for $G$; furthermore $\underline{\lambda}_j$ converges to some $\underline{\lambda}\in(4\Z[2])^n$. So passing to the limit we have found a standard-form basis ${\cal B}_\infty$ such that 
\[s_0 = e({\cal B}_\infty, \underline{\lambda}_{j-1})r_1({\cal B}_\infty)= e({\cal B}_\infty, \underline{\lambda}_{j-1})r_1(n, 0, \chi; {\cal X})s_1\cdots s_b\]  
Here we have written ${\cal B}_\infty={\cal X}\sqcup{\cal Y}$ where ${\cal X}$ denotes the first $n$ elements of ${\cal B}_\infty$. Now by Theorems \ref{CappingOff1} and \ref{CappingOff2} we find that $e({\cal B}_\infty, \underline{\lambda}_{j-1})r_1(n, 0, \chi; {\cal X})$ is in fact a Demushkin word in the free group $G'=\overline{\gp{\cal X}}$, and moreover the image of the orientation character of $G'$ equals the image of $\chi$. Therefore by Theorem \ref{DemusClass} there is a free basis ${\cal X}'$ of $G'$ such that 
\[e({\cal B}_\infty, \underline{\lambda}_{j-1})r_1(n, 0, \chi; {\cal X}) = r_1(n,0,\chi; {\cal X}') \]
Then if ${\cal B} = {\cal X}'\sqcup {\cal Y}$ then $\cal B$ is a basis of $G$ in standard form and $s_0=r_1(n,b,\chi;{\cal B})$ as required.
\end{proof}

%% file: complex.tex
\section{Splittings and the pro-$p$ curve complex}\label{secGimel}
An important tool in the study of surface automorphisms is the curve complex, which is a simplicial complex whose vertices are isotopy classes of essential simple closed curves on the surface. Translating into the language of group theory this may be rephrased as `a complex whose vertices are splittings of the surface group over \Z'. In this section we will define an analogous object for the pro-$p$ completion of a surface group, on which the automorphism group of the pro-$p$ group will act. One could no doubt study such an object for an arbitrary Demushkin group, but it is simpler and perhaps more interesting to restrict to the case of an orientable Demushkin group (i.e.\ the pro-$p$ completion of an orientable surface group).

For the rest of the section let $\Sigma$ be a compact orientable surface which is not a sphere, disc or cylinder. Let $\Gamma= \pi_1 \Sigma$ and let $G=\widehat{\Gamma}_{(p)}$ be the pro-$p$ completion of $\Gamma$. Let \famS\ be a family of subgroups of $G$ comprising the pro-$p$ completion of one representative of the fundamental group of each boundary component.
\begin{defn}
Define the {\em relative automorphism group} $\Aut_\bdy(G)$ to be the group of automorphisms of $G$ which send every $S\in \famS$ to a conjugate of itself. One may check that $\Aut_\bdy(G)$ is a closed subgroup of $\Aut(G)$ containing $\Inn(G)$. Define $\Out_\bdy(G)=\Aut_\bdy(G)/\Inn(G)$. Similarly for $\Gamma$. 
\end{defn}

\begin{defn}
Let $\gimel=\gimel(\Sigma)$ be the set of all equivalence classes of pairs $(T,\rho)$ with the following properties. 
\begin{itemize}
\item $T$ is a pro-$p$ tree and $\rho\colon G\to \Aut(T)$ is a continuous left action of $G$ on $T$ by graph automorphisms such that $G\lqt T$ is finite.
\item For every $e\in E(T)$ the edge stabiliser $\stab_\rho(e)$ is isomorphic to $\Z[p]$
\item If $G\lqt T$ is not a loop, then for every vertex $v\in V(T)$ and every edge $e$ incident to $v$ the edge stabiliser $\stab_\rho(e)$ is properly contained in the vertex stabiliser $\stab_\rho(v)$. 
\item Every $S\in\famS$ fixes a vertex of $T$, but does not fix any edge of $T$.
\end{itemize}
Two such objects are regarded as equivalent if there is a $G$-equivariant graph isomorphism from one to another. 

The automorphism group $\Aut_\bdy(G)$ acts on $\gimel(\Sigma)$ on the right by $(T,\rho)\cdot \psi=(T,\rho\circ \psi)$ where $\psi\in\Aut_\bdy(G)$. Note that if $\psi$ is an inner automorphism of $G$, then $(T,\rho)\cdot \psi$ is equivalent to $(T,\rho)$ so the action descends to an action of $\Out_\bdy(G)$ on $\gimel(\Sigma)$.
\end{defn}
The object $\gimel(\Sigma)$ will be referred to as the {\em pro-$p$ curve complex of $\Sigma$}. At present it is merely a set, but later we shall see that it has a topology making it into a profinite space such that the action of $\Aut_\bdy(G)$ is continuous. We will usually abuse notation by ignoring the fact that elements of $\gimel(\Sigma)$ are really equivalence classes and simply deal with the pairs $(T,\rho)$ themselves.

The definitions above parallel those for the abstract curve complex: indeed, if one replaces $G$ with $\Gamma$ and `pro-$p$ tree' with `simplicial tree' then one recovers the definition of the curve complex ${\cal C}(\Sigma)$ of $\Sigma$.

For our purposes we will adopt the following definition.
\begin{defn}
Given a set $X$, a {\em simplicial structure} on $X$ is a partial ordering $\leq$ on $X$ such that:
\begin{itemize}
\item For every $x\in X$ there exist only finitely many $y\in X$ such that $y\leq x$
\item Let $X^{(n)}$ be the set of elements $x\in X$ such that there are $n+1$ minimal elements of $X$ which are smaller than or equal to $x$. Then for each $x\in X^{(n)}$, $x$ is the join of these elements and the subposet $\{y\in X\mid y\leq x\}$ is isomorphic to the face poset of an $n$-simplex. We refer to $x\in X^{(n)}$ as an {\em $n$-simplex of $X$}.
\end{itemize} 
\end{defn}
The curve complex ${\cal C}(\Sigma)$ has a natural $\Aut_\bdy(\Gamma)$-equivariant simplicial structure which may be described as follows. For collections of isotopy classes of closed curves $x,y\in{\cal C}(\Sigma)$, we say $y\leq x$ if $y$ may be obtained by deleting some of the isotopy classes of closed curves in $x$. In terms of the action on dual trees, this corresponds to a $\Gamma$-equivariant epimorphism of trees which collapses some family of subtrees. Therefore we define a poset structure on $\gimel(\Sigma)$ as follows.
\begin{defn}
For $(T,\rho)$ and $(T',\rho')$ in $\gimel(\Sigma)$, we declare $(T,\rho)\geq (T',\rho')$ if and only if $T'$ is obtained from $T$ by collapsing some (equivariant) family of subtrees of $T$.
\end{defn}
An equivalent definition would be that for some family of disjoint connected full subgraphs $\{Y_i\}$ of $G\lqt T$, we obtain $T'$ from $T$ by collapsing each connected component of the pre-image of each $Y_i$ in $T$.
\begin{defn}
We define a {\em decomposition graph} of $\Sigma$ to be a finite connected graph $X$ with two labelling functions $n_\bullet, b_\bullet\colon V(X)\to \N$ such that:
\begin{itemize}
\item $\sum_{v\in V(X)} b_v$ is the number of boundary components of $\Sigma$
\item $\rk\Gamma= \sum_{v\in V(X)}(n_v+b_v)+2(1-\chi(X)) - \epsilon$ where $\epsilon$ is 1 if $\Sigma$ has boundary or 0 if $\Sigma$ is closed.
\item For a vertex $v$, if $n=0$ then either $X$ is a loop and $\Sigma$ is a torus or $b_v+\val(v)>2$
\end{itemize}
Let $D=D(\Sigma)$ be the set of decomposition graphs modulo isomorphism of labelled graphs.
\end{defn}

We will be aiming to prove the following theorem. Here we regard $\Aut_\bdy(\Gamma)$ as a subgroup of $\Aut_\bdy(G)$ in the natural way.
\begin{theorem}\label{ppCCThm}
Let $\Sigma$ be a compact orientable surface and let ${\cal C}(\Sigma)$ be the curve complex of $\Sigma$. There is an $\Aut_\bdy(\pi_1\Sigma)$-equivariant function $i\colon {\cal C}(\Sigma)\to \gimel(\Sigma)$ such that:
\begin{enumerate}[(1)]
\item $i$ is an injection. If $x,y\in{\cal C}(\Sigma)$ then $x\leq y$ if and only if $i(x)\leq i(y)$. Furthermore if $z\in \gimel(\Sigma)$ and $z\leq i(y)$ then $z=i(x)$ for some $x\leq y$.
\item $\leq$ is a simplicial structure on $\gimel(\Sigma)$ and $i$ preserves the simplicial structure
\item There is a natural diagram of bijections
\[\begin{tikzcd}[column sep = small]
{\cal C}(\Sigma)/\Aut_\bdy(\Gamma) \ar{rr}{i} \ar{dr}[swap]{\bar d} && \gimel(\Sigma)/\Aut_\bdy(G)  \ar{dl}{\bar \delta}  \\
&D(\Sigma)&
\end{tikzcd}\]
\end{enumerate}
\end{theorem}
\begin{rmk}
Part (3) of the theorem is perhaps the most interesting part, and may be paraphrased as `any splitting of $G$ with edge groups $\Z[p]$ is equivalent under $\Aut_\bdy(G)$ to a geometric splitting'. 
\end{rmk}
\begin{proof}[Construction of $i$]
Let $x$ be an $n$-simplex in the curve complex---that is, an isotopy class of collections of $n+1$ disjoint essential simple closed curves on $\Sigma$ which may be made disjoint. Such a collection (together with various choices of base-points) induces a graph-of-groups splitting of $\Gamma$ where all edge groups are copies of $\Z$, where peripheral subgroups are conjugate into vertex groups but not edge groups, and where edge groups are not equal to adjacent vertex groups except in the case of a torus with a single non-separating simple closed curve. 

This splitting is $p$-efficient by, for instance, \cite[Proposition 3.8]{Wilkes16}. Therefore by Propositions 6.5.3 and 6.5.4 of \cite{Ribes17} there is an induced splitting of $G$ as an injective graph of pro-$p$ groups with procyclic edge groups. The standard tree dual to this splitting (see \cite[Theorem 6.5.2]{Ribes17}) is a representative of an element $i(x)$ of $\gimel(\Sigma)$. Taking different choices of basepoints does not affect the equivalence class of $i(x)$ in $\gimel(\Sigma)$. Furthermore the abstract tree $T^{\rm abs}$ dual to the splitting of $\pi_1 \Sigma$ along $x$ is naturally embedded as an abstract subgraph of $T$ which is dense in the topology on $T$ by \cite[Proposition 6.5.4]{Ribes17}.
\end{proof}

\begin{proof}[Construction of $d$ and $\delta$]
There is a function $d\colon{\cal C}(\Sigma)\to D(\Sigma)$ defined as follows. Given a representative of an element of the curve complex (i.e.\ a collection $C$ of disjoint essential simple closed curves on $\Sigma$) there is a graph $X$ dual to these curves, with one vertex in each component of the complement of these curves and one edge for each curve. For a component $v$ of $\Sigma\smallsetminus C$ set $n_v$ to be half the genus of $v$ and $b_v$ to be the number of boundary components of $\Sigma$ lying in $v$. One easily verifies that this is a decomposition graph for $\Sigma$. This defines the map $d$. Clearly $d$ is $\Aut_\bdy(\Gamma)$-invariant and therefore gives a map $\overline{d}\colon {\cal C}(\Sigma)/\Aut_\bdy(\Gamma)\to D(\Sigma)$. By the standard theory of surfaces (see for example \cite[Chapter 1]{FM11}) this is a bijection.

We may also define a map $\delta\colon \gimel(\Sigma) \to D(\Sigma)$. Let $(T,\rho)$ represent an element of $\gimel(\Sigma)$. Then $G\lqt T=X$ is a finite connected graph. Choose a maximal subtree $Y$ of $X$ and a function $s\colon X\to T$ which is a section of the quotient map and such that $sd_0(x)=d_0(s(x))$ for all $x\in X$ and $sd_1(x)=d_1s(x)$ for all $x\in Y$. Such a map $s$ is a {\em fundamental 0-section} in the language of \cite{Ribes17}. Given $s$ we may form a graph of groups ${\cal G}= (X,G_\bullet)$ as in  \cite[Section 6.6]{Ribes17}, whose vertex and edge groups $G_x$ are the stabilisers of the points $s(x)$. This is an injective graph of pro-$p$ groups whose fundamental pro-$p$ group is $G$. Every $S\in\famS$ is conjugate into exactly one of the $G_v$. Replace every $S$ by some such conjugate and for $v\in V(X)$ let $\famS_v$ be the collection of elements of $S$ lying in $G_v$. Further let ${\cal E}_v$ be the collection of edge groups of $\cal G$ incident to $v$. Since $(G,\famS)$ is a PD$^2$ pair, \cite[Theorem 5.18]{Wilkes17} shows that $(G_v, \famS_v\cup {\cal E}_v)$ is also an orientable PD$^2$ pair for every $v\in V(x)$. Therefore $G_x$ is a free pro-$p$ group of rank $n_v+|\famS_v|+|{\cal E}_v|-1$ for some $n_v\in \N$. Set $b_v = |\famS_v|$. We have now defined a decomposition graph $\delta(T,\rho)$. 

One may also characterise these invariants as follows: $|{\cal E}_v|$ is the number of $G_v$-orbits of edges of $T$ incident to $s(v)$, and $b_v$ is the number of $G_v$-conjugacy classes of the conjugates of elements of $S$ that lie in $G_v$. Finally $n_v$ is defined as $\rk G_v-|\famS_v|-|{\cal E}_v|+1$ as before. This shows not only that $\delta(T,\rho)$ is independent of the choices of $Y$ and $s$ but is also $\Aut_\bdy(G)$-invariant. Hence we also have a map $\overline\delta\colon \gimel(\Sigma)/\Aut_\bdy(G)\to D(\Sigma)$. 
\end{proof}

\begin{proof}[Proof of (3)]
The first characterisation of $\delta$ in terms of graphs of groups shows immediately that $\delta\circ i = d$. This immediately shows that $\delta$ is a surjection. So to show that $i$ induces the required bijection it remains to show that $\overline{\delta}$ is an injection.

For a finite graph $X$ let $D_X(\Sigma)$ be the set of decomposition graphs whose underlying graph is $X$, and let $\gimel_X(\Sigma)$ be the set of elements $(T,\rho)$ of $\gimel(\Sigma)$ with $G\lqt T\iso X$. Then $\overline\delta$ breaks into a disjoint union of maps $\gimel_X(\Sigma)/\Aut_\bdy(G) \to D_X(\Sigma)$ and it suffices to prove that each of these is an injection. We do this by induction on the number of edges of $X$, over all surfaces $\Sigma$ simultaneously. The idea of the induction is simple. Each action on a tree may be represented by a graph of groups, which may be thought of as a smaller graph of groups plus one additional edge. The induction hypothesis allows us to identify the smaller graphs of groups, and we may then glue back up to recover the full decompositions. The precise proof requires much notation. 

Let $(T,\rho), (T', \rho')\in \gimel_X(\Sigma)$ and suppose $\delta(T,\rho)=\delta(T', \rho')$. Choose an edge $e$ of $X$ and let the component(s) of $X\smallsetminus e$ be $X_1$ (and $X_2$). If $e$ is separating we order the indices so that $d_0(e)\in X_1$. Further choose graphs of groups as above corresponding to $T$ and $T'$. Let the fundamental pro-$p$ groups of the restricted graphs of groups over $X_1$ (and $X_2$) be $G_1$ and $G_1'$ (and $G_2$ and $G_2'$ respectively). After replacing the $S\in\famS$ by conjugates if necessary, each member of \famS\ lies in exactly one of $G_1$  or $G_2$. Let $\famS_1$ consist of $d_0(G_e)$ (and $d_1(G_e)$ if $e$ is non-separating) together with those elements of $\famS$ lying in $G_1$. If $e$ is separating then set $\famS_2$ to be the elements of \famS\ lying in $G_2$ together with $d_1(G_e)$. Similarly define $\famS_i'$. Finally choose generators $s_0$ and $s_0'$ of $G_e$ and $G'_e$. If $e$ is non-separating denote a choice of stable letter for the HNN extension $G_1\amalg_{G_e}$ by $t$ so that $t^{-1}d_0(s_0)t=d_1(s_0)$. Similarly for $t'$.

By collapsing $X_1$ and $X_2$ and applying \cite[Theorem 5.18]{Wilkes17} we find that $(G_i,\famS_i)$ is a PD$^2$ pair. The rank and number of boundary components may be read off immediately from $\delta(T,\rho)=\delta(T',\rho')$. Both pairs are orientable so the classification of PD$^2$ pairs (Theorem \ref{MainThm}) shows that $(G_i,\famS_i)$ and $(G'_i,\famS'_i)$ are both isomorphic as pairs (after possibly changing the peripheral groups by conjugacies) to the pro-$p$ completion of the fundamental group of some compact surface $\widetilde\Sigma_i$.  

Now over $X_i$ we have graph of groups decompositions $U_i,U_i'\in \gimel_{X_i}(\widetilde\Sigma)$ of $(G_i,\famS_i)$ and $(G_i',\famS_i')$. Since $\delta(T,\rho)=\delta(T',\rho')$ and the graphs of groups are restrictions of those for $(G,\famS)$ it follows that $\delta(U)=\delta(U')$. By induction we therefore have an isomorphism $\psi_i\colon G_i\to G_i'$ (relative to $\famS_i$ and $\famS'_i$) such that $\psi_i(G_x)$ is a $G_i$-conjugate of $G'_x$ for each $x\in X_i$. Applying an inner automorphism we may assume $\psi_1(d_0(s_0))=d_0(s'_0)$ and, if $e$ is separating, that $\psi_2(d_1(s_0))=d_1(s'_0)$. When $e$ is separating we may therefore glue these maps together (by the universal property of amalgamated free products) to give $\psi\in\Aut_\bdy(G)$ such that $\psi(G_x)$ is a $G$-conjugate of $G'_x$ for all $x\in X$. Therefore the two graphs of groups decompositions $(X, \psi(G_\bullet))$ and $(X,G'_\bullet)$ differ only by the choice of 0-section, which does not affect the Bass-Serre tree (see \cite[Section 6.6]{Ribes17}). Therefore $(T,\rho)\cdot\psi=(T',\rho')$ as required.

It only remains to show that in the case when $e$ is separating we must show that the automorphism $\psi$ of $G_1$ extends across the HNN extension to give an automorphism of $G$. The construction above ensured that $\psi_1(d_0(s_0))=d_0(s'_0)$ and that $\psi_1(d_1(s_0))=(d_1(s'_0)^\lambda)^g$ for some $\lambda\in\U_p$ and $g\in G_1$. We must show that $\lambda=1$, for then $t\mapsto t'g$ gives the required extension of $\psi_1$ to $\psi\in\Aut_\bdy(G)$. We have a diagram of maps
\[\begin{tikzcd}
H_1(d_0(G_e)) \ar{dd}{{\rm conj}(t)}[swap]{\iso} \ar{rrr}{\psi_\ast}[swap]{\iso} & & & H_1(d_0(G'_e)) \ar{dd}{{\rm conj}(t')}[swap]{\iso} \\
& H_2(G_1,\famS_1) \ar{r}{\psi_\ast}[swap]{\iso} \ar[twoheadrightarrow]{lu} \ar[twoheadrightarrow]{ld} & H_2(G'_1,\famS'_1) \ar[twoheadrightarrow]{ru} \ar[twoheadrightarrow]{rd} & \\
H_1(d_1(G_e)) \ar{rrr}{{\rm conj}(g^{-1})\circ \psi_\ast}[swap]{\iso} &  & & H_1(d_1(G'_e))
\end{tikzcd}\]
with coefficients in $\Z[p]$. Here ${\rm conj}(x)$ denotes a map induced by conjugacy by $x\in G$. The two triangles in this diagram commute because of the segment
\[H_2(G,\famS_1) \longrightarrow H_2(G_1,\famS_1) \stackrel{\star}{\longrightarrow} H_1(G_e) \]
of the (dual of) the long exact sequence in \cite[Theorem 4.11]{Wilkes17}, where the starred map is the composition
\[H_2(G_1,\famS_1)\to H_1(\famS_1) \to H_1(d_0(G_e))\oplus H_1(d_1(G_e)) \xrightarrow{d_1^{-1}-d_0^{-1}} H_1 (G_e)  \]
Therefore the entire diagram commutes. Tracing $d_0(s_0)$ around the outer rectangle shows that $\lambda=1$ as required.
\end{proof} 
\begin{proof}[Proof of (1)]
Let us prove that $i$ is an injection. Suppose $i(x)=i(y)$. In order to show $x=y$ it suffices to show that the edge stabilisers of $x$ and $y$ are pairwise conjugate in $\Gamma$, since a conjugacy class in $\Gamma$ determines a simple closed curve on $\Sigma$ up to isotopy, and a family of simple closed curves determines a splitting in ${\cal C}(\Sigma)$. Let graphs of groups corresponding to $x$ and $y$ be $(X, \Gamma_\bullet)$ and $(Y, \Gamma'_\bullet)$ respectively. For each edge in $X$ (or $Y$) choose a generator $x_e$ for $\Gamma_e$ (respectively $y_e\in \Gamma'_e$). The isomorphism of profinite trees witnessing $i(x)=i(y)$ gives an isomorphism $f\colon X\to Y$ such that $x_e$ is conjugate in $G$ to $y_f(e)^\kappa$ for some $\kappa\in\Z[p]^{\!\times}$. Since $\Gamma$ is conjugacy $p$-separable \cite{Paris09, Wilkes16}, to prove that $x_e$ and $y_{f(e)}$ are conjugate in $\Gamma$ it suffices to prove that $\kappa=\pm 1$.

If $e$ is a non-separating edge then there exists a homomorphism $\phi\colon G\to \Z[p]$ sending $\Gamma$ to $\Z$ and $x_e$ to 1. Since $\phi(x_e)$ is central the image of $\phi(y_{f(e)})=\kappa^{-1}$; which was in $\Z$ by definition of $\phi$. By symmetry we find $\kappa\in \Z$ also; so $\kappa\in\Z^{\!\times}=\{\pm 1\}$. If $e$ is separating then there exists a map to the group of upper triangular $3\times 3$ matrices with entries in $\Z[p]$ which maps $\Gamma$ to integer matrices and sends $x_e$ to a central element; the same argument yields $\kappa=\pm 1$.

This concludes the proof that $i$ is an injection.


To show that $x\leq y$ implies $i(x)\leq i(y)$ note that the graph of groups decomposition of $\Gamma$ corresponding to $x$ is obtained from the decomposition corresponding to $y$ by collapsing some sub-graphs-of-groups; the same collapses exhibit $i(x)\leq i(y)$.

Also, if $z\leq i(y)$ then at the level of graphs of groups $z$ is obtained from $i(y)$ by collapsing some sub-graphs-of-groups; such collapses are precisely those induced by collapses of the graph of discrete groups corresponding to $y$. This shows that $z=i(x)$ for some (unique) $x\leq y$.
\end{proof}

\begin{proof}[Proof of (2)]
Part (2) follows from (1) and (3). For the notion of simplicial structure consists of conditions on $\{y\in \gimel(\Sigma)\mid y\leq x\}$ for $x\in\gimel(\Sigma)$. Part (1) and the simplicial structure on ${\cal C}(\Sigma)$ show that these conditions hold when $x\in i({\cal C}(\Sigma))$. Since the partial order is $\Aut_\bdy(G)$-invariant and every point in $\gimel(\Sigma)$ may be translated to a point in $i({\cal C}(\Sigma))$ via the action of $\Aut_\bdy(G)$ this shows that the simplicial structure conditions hold at every point of $\gimel(\Sigma)$. This proves (2).
\end{proof}

Notice that $\gimel(\Sigma)$ has a natural topology making it into a profinite space on which $\Aut_\bdy(G)$ acts continuously. For choosing any lift $s\colon D(\Sigma)\to \gimel(\Sigma)$ splitting the quotient map defines a surjection $\Aut_\bdy(G)\times D(\Sigma)\to \gimel(\Sigma)$, to which we may give the quotient topology. This is homeomorphic to the disjoint union
\[\gimel(\Sigma) = \bigsqcup_{z\in D(\Sigma)} \stab(s(z)) \lqt\Aut_\bdy(G)\]
so that $\gimel(\Sigma)$ is a profinite space. Because $D(\Sigma)$ is finite this topology does not depend on the choice of $s$.

%% file: p-congtop.tex
\section{The $p$-congruence topology on the mapping class group}\label{secpCongTop}
In this section $\Sigma$ will be a closed orientable surface of genus at least 2.

We conclude this paper with some remarks about separability of certain subgroups of the mapping class group and their relation to the curve complex. A `separability property' of a subgroup $\Delta$ in some group $\Gamma$ relates to closedness in some profinite topology on $\Gamma$. The standard notion of separability concerns the full profinite topology on $\Gamma$, whose basic open sets consist of all cosets of finite index subgroups of $\Gamma$---{\it i.e.}\ the topology induced from the map from $\Gamma$ into its profinite completion. Separability of various subgroups of $\Out(\Gamma)$ with respect to the full profinite topology were considered in \cite{LM07}. One may also consider for instance the pro-$p$ topology on $\Gamma$ to give the notion of $p$-separability. 

In the context of this paper it is natural to discuss the topology $\cal T$ on $\Out(\Gamma)$ induced by the natural map into $\Out(G)$. One may refer to this as the {\em $p$-congruence topology}, since it concerns `$p$-congruence subgroups of $\Out(\Gamma)$'---the kernels of maps to $\Out(P)$ where $P$ is a characteristic $p$-group quotient of $\Gamma$. This is not quite a pro-$p$ topology, but is virtually pro-$p$ (see \cite[Proposition 5.5]{DDMS}). Note that this topology is rather weaker than the full profinite topology.

\begin{theorem}
The $p$-congruence topology on $\Out(\Gamma)$ is Hausdorff and stabilisers of multicurves are closed.
\end{theorem}
\begin{proof}
The fact that this topology is Hausdorff is exactly the statement that $\Out(\Gamma)$ embeds into $\Out(G)$. This was shown in \cite{Paris09}, but also follows from the faithfulness of the action of $\Out(\Gamma)$ on ${\cal C}(\Sigma)$ and Theorem \ref{ppCCThm}. 

`Multicurve stabilisers' are the stabilisers of finitely many conjugacy classes of cyclic subgroups in $\Gamma$ coming from simple closed curves on $\Sigma$. Such a subgroup $\Delta$ is a stabiliser of a point $x\in {\cal C}(\Sigma)$. If $g\in \Gamma\smallsetminus\Delta$ then $g\cdot x\neq x$, hence $g\cdot i(x)\neq i(x)$. Thus $g$ lies outside the closed subgroup ${\rm stab}(i(x))$ of $\Out(G)$, hence lies outside the closure of $\Delta$ in $\Gamma$.

One may also consider the stabilisers of finitely many conjugacy classes of elements $g_i$ in $\Gamma$ corresponding to simple closed curves on $\Sigma$---that is, to stabilise the multicurve fixing a choice of orientation on the curves. These groups are also separable. It only remains to separate out those outer automorphisms sending some $g_i$ to a conjugate of $g_j^{-1}$ from those sending $g_i$ to a conjugate of $g_j$. But all such $g_j$, which are given by simple closed curves, are distinguished from their inverses by the map from $\Gamma$ to $\Gamma / \Gamma_4$, where $\Gamma_4$ is the fourth term in the lower central $p$-series. (If $p\neq 2$ then $\Gamma/\Gamma_3$ suffices.) Hence the $p$-congruence topology separates these subgroups too.
\end{proof}

For the remainder of the paper we will consider $\Out^+(\Gamma)$, which coincides with the mapping class group ${\rm MCG}(\Sigma)$ of $\Sigma$ \cite[Theorem 8.1]{FM11} and consists of automorphisms which are `orientation-preserving'. This is an open subgroup of $\Out(\Gamma)$ in the $p$-congruence topology, consisting of those automorphisms acting trivially on $H^2(G; \Z/p^2)$.

Leininger and McReynolds \cite{LM07} also consider separability of the `geometric subgroups' of $\Out^+(\Gamma)$ defined by Paris and Rolfsen \cite{PR99}, which are those deriving from subsurfaces of $\Sigma$. Let $\Sigma'$ be an incompressible proper subsurface of $\Sigma$, where `incompressible' means that each component of $\Sigma'$ is $\pi_1$-injective. The geometric subgroup ${\cal G}(\Sigma')$ of ${\rm MCG}(\Sigma)$ corresponding to $\Sigma'$ consists of those self-diffeomorphisms of $\Sigma$ isotopic to a map fixing $\overline{\Sigma\smallsetminus\Sigma'}$ pointwise. As shown in \cite{PR99} this subgroup is in fact the image of the mapping class group of $\Sigma'$ under the obvious homomorphism.

We claim that geometric subgroups are closed in the $p$-congruence topology. Indeed, the geometric subgroup consists of the intersection of finitely many subgroups of the form $\Delta_g$,  where $\Delta_g$ is the set of outer automorphisms fixing the conjugacy class of $g$. Equivalently $\Delta_g$ consists of those self-homeomorphisms of $\Sigma$ which, up to isotopy, fix a simple closed curve corresponding to $g$ pointwise. We have already seen that such subgroups $\Delta_g$ are separable, so it follows that the geometric subgroup is also separable.

Consider the following family of simple closed curves on $\Sigma$:
\begin{itemize}
\item the boundary components $b_i$ of components of $\Sigma'$;
\item curves $l_i$ decomposing the components of $\overline{\Sigma\smallsetminus\Sigma'}$ into pairs-of-pants;
\item for each $i$ a curve $k_i$ transverse to $l_i$, meeting it exactly once, and disjoint from all other $l_j$ and $b_j$.
\end{itemize}
We claim that the geometric subgroup is equal to the intersection of the groups $\Delta_g$ corresponding to these curves. Suppose a self-diffeomorphism $\phi$ of $\Sigma$ fixes all the $b_i$, $l_i$ and $k_i$ pointwise. Since $\phi$ is orientation-preserving, it preserves the components of $\Sigma\smallsetminus\{b_i, l_i\}$. these components either lie in $\Sigma'$ or are annuli or pairs of pants. Since the mapping class group of a pair of pants (or annulus) is abelian and generated by Dehn twists about the boundary components \cite[Section 3.6.4]{FM11}, $\phi$ is isotopic to a product of Dehn twists in the $b_i$ and $l_i$ (and these Dehn twists commute).  If this product involves a Dehn twist about some $l_i$, then $\phi$ cannot possibly fix $k_i$. It follows that $\phi$ is a product of Dehn twists in the $b_i$. By an isotopy supported near the $b_i$ we may now `push these Dehn twists into $\Sigma'$' and find that $\phi\in {\cal G}(\Sigma')$. Hence:
\begin{theorem}
Geometric subgroups of ${\rm MCG}(\Sigma)$ are closed in the $p$-congruence topology.
\end{theorem}